\documentclass[12pt]{article}   

\usepackage{amssymb}
\usepackage{mathbbol}
\usepackage{latexsym}
\usepackage{amsthm}
\usepackage{enumerate}
\usepackage{epsfig}
\usepackage{graphicx}
\usepackage{color}
\usepackage{float}
\usepackage{subfigure}
\usepackage{amsmath}
\usepackage{makeidx}
\usepackage{fancyhdr}
\usepackage{lastpage}
\usepackage{url}
\usepackage{setspace}
\newcommand{\mb}{\mathbb}
\newcommand{\mc}{\mathcal}
\newcommand{\m}{\mbox}
\newcommand{\f}{\frac}
\newcommand{\ld}{\lambda}
\newcommand{\og}{\omega}

\newcommand{\rt}{\sqrt}
\newcommand{\fa}{\forall}

\newcommand{\dd}{\partial}
\newcommand{\tb}{\textbf}

\newcommand{\nn}{\nonumber}
\newcommand{\ep}{\epsilon}
\newcommand{\qd}{\quad}
\newcommand{\al}{\alpha}

\newcommand{\iy}{\infty}

\DeclareSymbolFont{AMSb}{U}{msb}{m}{n}  
\DeclareMathSymbol{\Sph}{\mathbin}{AMSb}{"53} \DeclareMathSymbol{\R}{\mathbin}{AMSb}{"52}
\DeclareMathSymbol{\T}{\mathbin}{AMSb}{"54} \DeclareMathSymbol{\Z}{\mathbin}{AMSb}{"5A}
\DeclareMathSymbol{\K}{\mathbin}{AMSb}{"4B}

\newenvironment{theorem}[1]{\vspace{0.2cm}\noindent    {\bf Theorem {#1}}}{\vspace{.1cm}}
\newenvironment{lemma}[1]{\vspace{0.2cm}\noindent    {\bf Lemma {#1}}}{\vspace{.1cm}}
\newenvironment{corollary}[1]{\vspace{.9cm}\noindent    {\bf Corollary {#1}}}{\vspace{.1cm}}

\def\qed{\hfill $\Box$}
\renewenvironment{proof}{\vspace{-.05cm}   \noindent{\bf Proof: }}{\qed \vspace{0.75cm}}

\newcounter{case}

\begin{document}
	
	\title{Stability estimates for relativistic Schr\"odinger equation from partial boundary data}
	
	\author{Soumen Senapati\footnote{\emph{E-mail address:} soumen@tifrbng.res.in}\\
	\multicolumn{1}{p{.4\textwidth}}{\centering{\small{Tata Institute of Fundamental Research Centre for Applicable Mathematics }   
			Bangalore, India}}}
	\maketitle
	
	\abstract{In this article we study stability aspects for the determination of time-dependent vector and scalar potentials in relativistic  Schr\"odinger equation from partial knowledge of boundary measurements. For space dimensions strictly greater than 2 we obtain log-log stability estimates for the determination of vector potentials (modulo gauge equivalence) and log-log-log stability estimates for the determination of scalar potentials from partial boundary data assuming suitable a-priori bounds on these potentials.}
	
	\vspace{2mm}
	
		
		
	\section{Introduction and statement of the main result}
	\label{sec:1}
	Let $\Omega$ be a bounded domain in $\R^n\ (n\ge3)$ with smooth boundary $\Gamma$.  Let $\nu(x)$ be the outward unit normal to $\Gamma$ and $Q$ be the finite cylindrical domain defined by $Q=(0,T)\times\Omega$ where $T>$ diam$(\Omega)$. We denote lateral boundary of $Q$ by $\Sigma:=(0,T)\times\Gamma$. The relativistic Schr\"odinger operator on $Q$ denoted by $\mc{L}_{\mc{A},q}$ is defined as	
	\[\mc{L}_{\mc{A},q}=(\dd_t+A_0(t,x))^2-\sum_{k=1}^{n}{(\dd_{x_k}+A_k(t,x))^2}+q(t,x).\]
	Here $\mc{A}\equiv(A_i)_{0\le i\le n}$ is the vector potential and $q$ is the scalar potential. We assume $q\in\mc{L}^{\infty}(Q)$ and all the components of vector potential are real valued functions belonging to $C_c^{\iy}(Q)$. We are interested in deriving stability estimates for the recovery of $\mc{A}$ (upto gauge equivalence) and $q$ from partial boundary data.\\
	
	\noindent We consider the IBVP:
	\[ 
	\begin{cases}
	\mc{L}_{\mc{A},q}(u)=0 \m{ in }Q,\\
	\left(u|_{t=0}, \partial_{t}u|_{t=0},u|_{\Sigma}\right)=(u_0,u_1,f).
	\end{cases}
	\]
	From \cite{lions,lassas} it is well known that if  $u_0\in H^1(\Omega)$, $u_1\in L^2(\Omega)$ and $f\in H^1(\Sigma)$ with the compatibility criteria, $u_0|_{\Gamma}=f|_{\{0\}\times\Gamma}$ then the above IBVP has a unique solution in $C^1\big([0,T];L^2(\Omega)\big)\cap C\big([0,T];H^1(\Omega)\big)$ and there exists a constant $C>0$ such that for any $t\in[0,T]$ we have
	\begin{align}
	\|\dd_\nu u\|_{L^2(\Sigma)}+\|\dd_t u(t,\cdot)\|_{L^2(\Omega)}+\|u(t,\cdot)\|_{H^1(\Omega)}\le C\left(\|u_0\|_{H^1(\Omega)}+\|u_1\|_{L^2(\Omega)}+\|f\|_{L^2(\Sigma)}\right).\label{regularity}
	\end{align}
	Let us introduce few notations before we state the result.
	For $\og\in\mb{S}^{n-1}$, we define the following subsets of $\Gamma$ and $\Sigma$, respectively
	\begin{align}
	&\Gamma_{+}(\omega)=\{x\in\Gamma; \nu(x)\cdot\omega>0\}, \qd \Gamma_{-}(\omega)=\{x\in\Gamma;\nu(x)\cdot\omega<0\},\nn\\
	\label{} &\Gamma_{+,\ep}(\omega)=\{x\in\Gamma;\ \nu(x)\cdot\omega>\ep\},\qd \Gamma_{-,\ep}(\omega)=\{x\in\Gamma;\ \nu(x)\cdot\omega<\ep\},\nn\\
	\label{Gammaplusminus}  &\Sigma_{\pm}(\omega)=(0,T)\times \Gamma_{\pm}(\omega),\qd \Sigma_{\pm,\epsilon/2}(\omega)=(0,T)\times \Gamma_{\pm,\epsilon/2}(\omega).
	\end{align}	
	For a fixed  $\og_0\in\mb{S}^{n-1}$, let us define the input-output operator $\Lambda$ as 
	\begin{align}
	&\Lambda:{H^1(\Omega)\times L^2(\Omega)\times H^1(\Sigma)}\mapsto\ {L^2(\Sigma_{-,\epsilon/2}{(\og_0)})\times H^1(\Omega)} \label{defn}\\
	&\Lambda(u_0,u_1,f)=\left(\dd_\nu u|_{\Sigma_{-,\epsilon/2}{(\og_0)}},u(T,\cdot)\right).\label{Defn}
	\end{align}
	By $\|\cdot\|_{*}$ we denote operator norm of the input-output operator with respect to the range and domain as indicated in \eqref{defn}. Given $C_0\m{ and }\alpha>0$ we introduce the admissble set of potentials $(\mc{A},q)$ as 
	\[\mc{M}(C_0,\alpha)=\{(\mc{A},q)\in C_c^{\iy}(Q)^{n+1}\times L^{\iy}(Q);\  \|\mc{A}\|_{H^{\f{n+1}{2}+\alpha}(Q)}\le C_0,\|q\|_{H^{\f{n+1}{2}+\alpha}(Q)}\le C_0\}.\]
	Now we state the main result.
	
	\begin{theorem}{} 
		For $i=1,2\ $ let $(\mc{A}_i,q_i)\in\mc{M}({C_0,\alpha})$ and $T>\m{diam}(\Omega)$. We denote the input-output operator corresponding to $\mc{L}_{\mc{A}_i,q_i}$ by $\Lambda_i$. Further,
		assume that div$_{(t,x)}\mc{A}_1$=div$_{(t,x)}\mc{A}_2$. Then there exists $C,\mu_1,\mu_2,\alpha_1$ and $\alpha_2>0$ depending on $C_0,\alpha$ and $Q$ such that
		
		\vspace*{-1mm}
		
		\begin{align} 
		&\|\mc{A}_2-\mc{A}_1\|_{L^{\infty}(Q)}\le C\left(\|\Lambda_1-\Lambda_2\|_{*}^{\mu_1}+\big|\log|\log\|\Lambda_1-\Lambda_2\|_{*}|\big|^{-\mu_2}\right),\label{main result 1}\\
		&\|q_2-q_1\|_{L^{\infty}(Q)}\le C\left(\|\Lambda_1-\Lambda_2\|_{*}^{\alpha_1}+\Big|\log\big|\log|\log\|\Lambda_1-\Lambda_2\|_{*}|\big|\Big|^{-\alpha_2}\right).\label{main result 2}
		\end{align}	
	\end{theorem}

	Let us give some prior work done in the context of inverse problems for hyperbolic equations. Motivated by the construction of complex geometric optics solutions by Sylvester and Uhlman\cite{uhlman}, Rakesh and Symes\cite{raksymm} proved unique determination of time-independent scalar potential in the wave equation from full Neumann-Dirichlet data. It was extended by Isakov \cite{isa1} to the recovery of time-independent time derivative perturbation but in the absence of any space derivative perturbation. In \cite{rammsjo}, Ramm and Sj\"ostrand dealt uniqueness issues of time-dependent potential  in an infinte cylinder. This was later generalized by Salazar; see \cite{salazar} and \cite{salazar 2}. For finite time, Rakesh and Ramm showed in \cite{rakramm} that time-dependent potential can be recovered in some specific set outside which they are known. We should also mention the unique recovery of time dependent potential from scattering data by Stefanov\cite{Stefanov 1}. Kian in \cite{kian1},\cite{kian2},\cite{kian3} considered the problem of unique determination and stability of time derivative perturbation and scalar potential from full Dirichlet to Neumann data. Using properties of light-ray transform from \cite{siamak}, \cite{Stefanov 2}  Krishnan and Vashisth\cite{VM} proved uniqueness of all coefficients (upto a gauge invariance for vector potential term) appearing in relativistic Schr\"odinger equation from partial boundary data. Similar coefficient recovery problems in various settings were extensively studied by Yamamoto, Bellassoued, Choulli and Ben A\"icha etc in numerous papers; see \cite{BJY1},\cite{BJY2},\cite{bellaicha1},\cite{bellaicha1},\cite{aicha1},\cite{aicha2},\cite{BR}. Bellassoued and Ben A\"icha \cite{bellaicha1} stably recovered both time-dependent vector field term and scalar potential from full input-output operator but in the absence of time-derivative perturbation. In a recent work by Bellassoued and Fraj\cite{BF}, using Neumann measurements made on arbitrary part of the boundary stable determination of zeroth order time-dependent perturbation was shown. The current paper strengthens the result by Bellassoued and Ben A\"icha\cite{bellaicha1} even in the full data case. In all of our discussion, we consider smooth coefficients vanishing on boundary for vector potentials for simplicity. One can use the approximation argument presented in \cite{kian1} for more general coefficients. 
	
	\section{Carleman Estimates and Geometric Optics Solutions}
	
	We start by providing geometric optics solutions to the relativistic Schr\"odinger equation depending on a large parameter. 	
	The existence of geometrical optics solutions will be shown using weighted ${L^2}$-coercivity of some conjugated operators also known as Carleman estimate. Then to bound certain boundary terms we will need boundary Carleman estimate. In \cite{kian2}, Kian proved those estimates. Without proof we state the results by Krishnan and Vashisth\cite{VM} which were motivated by the one in \cite{kian2}.    
	
	\begin{theorem}{(Boundary Carleman estimates)}
	For $\og\in\mb{S}^{n-1}$, $(\mc{A},q)\in{\mc M(C_0,\alpha)}$ and $u\in C^2(\bar{Q})$ satisfying $u|_{t=0}=\partial_tu|_{t=0}=\ u|_{\Sigma}=0\ $ there exists $\ld_0,C>0$ both of which depend only on $C_0,\alpha\m{ and }Q$ such that for  $\ld\ge\ld_0$ we obtain
	\begin{align}
	\int_{Q}  e^{-2\ld(t+x\cdot \og)}\left(\ld^2|u(t,x)|^2+|\nabla_{(t,x)}u(t,x)|^2\right) dxdt&+\ld \int_{\Sigma_{+}(\omega)}e^{-2\ld(t+x\cdot \og)}|\omega\cdot \nu(x)||\partial_{\nu}u|^2dS\nn\\
	+\ld\int_{\Omega} e^{-2\ld(T+x\cdot \og)}|\dd_tu(T,x)|^2\ dx\le C\ \Bigg{(}\int_{Q}& e^{-2\lambda(t+x\cdot\omega)}|\mathcal{L}_{\mathcal{A},q}u(t,x)|^2dxdt\nn\\ 
	+\int_{\Omega} &e^{-2\lambda(T+x\cdot\omega)}\left(\lambda^2|u(T,x)|^2+\ld|\nabla_x{u(T,x)}|^2\right)dx\nn\\
	+&\lambda\int_{\Sigma_{-}(\omega)}{ e^{-2\lambda(t+x\cdot\omega)}|\omega\cdot \nu(x)||\partial_{\nu}u|^2}dS \Bigg{)}.\label{boundarycarleman}
	\end{align}
    \end{theorem}

	\begin{corollary}{1.} (\tb{Interior Carleman estimates})
		Given $(\mc{A},q)\in \mc{M}(C_0,\alpha)$ then there exist $C>0,\ld_0>0$ depending only on $C_0,\alpha$ and $Q$ such that the following estimate holds for $u\in C_c^{\iy}(Q)$ and $\ld\ge\ld_0$;
		\begin{align}
		\int_{Q} e^{-2\ld(t+x\cdot \og)}(\ld^2|u(t,x)|^2+|\nabla_{(t,x)}u(t,x)|^2)dxdt\le C\|e^{-\ld(t+x\cdot\og)}\mc{L}_{\mc{A},q}u\|_{{L}^2(Q)}^2.\nn
		\end{align}
	\end{corollary}
	\subsection{Construction of Geometric Optics Solutions}
	\label{subsec:2}  
	We make use of interior Carleman estimates and Hahn-Banach extension theorem to find the following parameter dependent solutions.

	\begin{theorem} {}
		Let $\phi\in C_c^{\iy}(\mb{R}^n)$ and $(\mc{A},q)\in \mc{M}(C_0,\alpha)$. Then there exists $C>0,\ld_0>0$ depending on $\phi,\ C_0,\ \alpha$ and $Q$ such that for $\ld\ge\ld_0$ 
		\[u(t,x)= e^{\lambda(t+x\cdot\og)}\big(\phi(x+t\og)   e^{-\int_{0}^{t}(1,-\omega)\cdot\mathcal{A}(s,x+(t-s)\omega) ds}+R_\lambda(t,x)\big)\] solves $\mathcal{L}_{\mathcal{A},q} u=0$, where $ \|R_\lambda\|_{H^k(Q)}\le C\lambda^{-1+k}\|{\phi}\|_{H^3(\mathbb{R}^n)}$ for $k\in\{0,1,2\}$.
	\end{theorem}
	
	\vspace*{4mm}
	
	\begin{proof}		
		Following exactly same set of arguments presented in \cite{kian1}, we can have similar Carleman estimates in negative order Sobolev spaces but with an additional index shift by -1. We state the result below. For a proof see \cite{kian1}.
		
		\begin{lemma}{1}{\textnormal{ (Lemma 5.4 of \cite{kian1})}}	
			For $(\mc{A},q)\in \mc{M}(C_0,\alpha)$ let us consider the conjugated operator $$P_{\mc{A},\ld,\og}=e^{-\ld(x\cdot\og+t)}(\mathcal{L}_{\mathcal{A},q}-q)e^{\ld(x\cdot\og+t)}.$$ There exists $C,\ld_0>0$ depending on $C_0,\alpha$ and $Q$ such that for $u\in C_c^{\iy}(Q)$ and $\ld\ge\ld_0$
			\begin{align*}
			\|u\|_{H^{-1}_{\ld}(\mb{R}^{n+1})}\le C\|P_{\mc{A},\ld,\og}u\|_{H^{-2}_{\ld}(\mb{R}^{n+1})}.
			\end{align*} 	
		\end{lemma}
		
		
		\noindent 
		For $f\in H^1_\ld(Q)$, we define the linear map $\mc{S}:\{P_{-\mc{A},-\ld,\og}v;\ v\in C_c^{\iy}(Q)\}\to\R$ by
		\begin{align}
		\mc{S}(P_{-\mc{A},-\ld,\og}v)=\int_{Q}vf\ dxdt.\label{definition}
		\end{align}
		Now we use the above Lemma to conclude $\mc{S}$ is continuous that is,
		\begin{align}
		|\mc{S}(P_{-\mc{A},-\ld,\og}v)|\le \|P_{-\mc{A},-\ld,\og}v\|_{H^{-2}_{\ld}(\mb{R}^{n+1})}\|f\|_{H^{1}_{\ld}(Q)}.
		\end{align}
		Now by Hahn-Banach theorem we can extend $\mc{S}$ as a continuous functional on $H^{-2}_{\ld}(\mb{R}^{n+1})$ still denoted as $\mc{S}$ and satisfies $\|\mc{S}\|\le\|f\|_{H^{1}_{\ld}(Q)} $. But by Reis\'z representation theorem, we have a unique $u\in{H^{2}_{\ld}(\R^{n+1})}$ such that the following holds  
		\begin{align}
		\mc{S}(v)=(v,u)_{{H^{-2}_{\ld}(\R^{n+1})},{H^{2}_{\ld}(\R^{n+1})}}.\label{exist1}
		\end{align}
		Combining \eqref{definition} and \eqref{exist1}  we get for all $v\in C_c^{\iy}(Q)$
		\begin{align}
		\mc{S}(P_{-\mc{A},-\ld,\og}v)=\int_{Q}vf\ dxdt=(P_{-\mc{A},-\ld,\og}v,u)_{{H^{-2}_{\ld}(\R^{n+1})},{H^{2}_{\ld}(\R^{n+1})}}.\label{exist2}
		\end{align}
		From \eqref{exist2},
		\begin{align}
		u\in{H^{2}_{\ld}(\R^{n+1})}\m{ satisfies }P_{\mc{A},\ld,\og}u=f\m{ with  }\|u\|_{{H^{2}_{\ld}(\R^{n+1})}}\le \|\mc{S}\|\le\|f\|_{H^{1}_{\ld}(Q)}.\nn
		\end{align}
		Now we want solutions of relativistic Schr\"odinger equation of the following form
		\begin{align}
		&\qd \qd u(t,x)=e^{\ld(x\cdot\og+t)}(B(t,x)+R_{\ld}(t,x)).\nn
		\end{align}
		That is, we should find $B$ and $R_{\ld}$ so that, $ e^{-\ld(x\cdot\og+t)}\mathcal{L}_{\mathcal{A},q}e^{\ld(x\cdot\og+t)}(B(t,x)+R_{\ld}(t,x))=0,$ 
		
		\vspace{2mm}
		
		\noindent or, $ P_{\mathcal{A},\ld,\og}R_{\ld}=-qR_{\ld}-P_{\mathcal{A},\ld,\og}B-qB=-qR_{\ld}-2\ld(1,-\og)\cdot(\nabla_{(t,x)}B+\mc{A}B)-\mc{L}_{\mc{A},q}B.$
		
		\vspace{2mm}
		
		\noindent If we take  $B(t,x)=\phi(x+t\og)e^{-\int_{0}^{t}(1,-\og)\cdot\mc{A}(s,x+(t-s)\og)ds}$, it satisfies
		\[(1,-\og)\cdot(\nabla_{(t,x)}B+\mc{A}B)(t,x)=0.\]		
		Since $\mc{L}_{\mc{A},q}B\in H^1_\ld(Q)$ with $\|\mc{L}_{\mc{A},q}B\|_{H^1_\ld(Q)}\le C\ld\|\phi\|_{H^3(\R^n)}$ it suffices to find 
		\begin{align} R_\ld\in {H^2(Q)} \m{  satisfying } (P_{\mathcal{A},\ld,\og}+q)R_{\ld}=-\mc{L}_{\mc{A},q}B.\label{exist3}
		\end{align}
		Define $T(f)=u$ where $P_{\mc{A},\ld,\og}u=f$. Then $\|T\|_{H^1_\ld(Q)\to H^1_\ld(Q)}\le \f{C}{\ld}$.\\
		Thus the problem \eqref{exist3} reduces to finding $\tilde{f}\in H^1_\ld(Q)$ such that \[(I+qT)\tilde{f}=-\mc{L}_{\mc{A},q}B.\]
		For large enough $\ld$, we have invertibility of $(I+qT)$ in $H^1_\ld(Q)$. So, we can find $\tilde{f}\in H^1_\ld(Q)$ and hence $R_\ld\in H^2_\ld(Q)$ which satisfies 
		\begin{align} 
		\label{remainder}\|R_\ld\|_{H^2_\ld(Q)}\le C\|\tilde{f}\|_{H^1_\ld(Q)}\le C\ld\|\phi\|_{H^3(Q)}.
		\end{align}
		For future purposes we write \eqref{remainder} in a different way which is, 
		\[\|R_\lambda\|_{H^k(Q)}\le C\lambda^{-1+k}\|{\phi}\|_{H^3(\mathbb{R}^n)} \m{ for }\ k\in\{0,1,2\}.\]
		This ends the construction of geometric optics solutions.
	\end{proof}
	
	\section{Proof of the main theorem}
	\label{sec:3}
	\subsection{{Stability estimate for vector potential}} \label{subsec:1}%
	We outline the proof as follows. Using Green's formula and geometric optics solutions constructed earlier we will establish estimates connecting vector potential and input-output operator, while doing so we will crucially use boundary Carleman estimate. Then we estimate the line integrals of a component of the vector potential along the direction of light-rays by the input-output operator. We include these in Lemma 3 and end up deriving a Fourier estimate as a corollary. Then under the divergence free condition of vector potential we will stably recover all components of the vector potential. We use Vessella's conditional stability result \cite{vassella}  to obtain Fourier estimate for small frequencies.\\

	\begin{lemma}{2.}{ \bf{(Integral identity and estimates)}} \label{integralidentityestimates}
		For $i=1,2$ let $(\mc{A}_i,q_i)\in \mc{M}(C_0,\alpha)$ with $\phi\in C_c^{\iy}(\mb{R}^n)$ and $\ \og\in\mb{S}^{n-1}$ satisfy    $|\og-\og_0|<\f{\epsilon}{2}$. Then there exists $\beta>0,C>0,\ld_0>0$ which depend on $C_0,\alpha$ and $Q$ such that for all $\lambda\ge\lambda_0$
		\begin{align}
		\left\vert\int_{0}^{T}\int_{\mathbb{R}^n}(1,-\omega)\cdot(\mathcal{A}_2-\mc{A}_1)(t,y-t\omega) e^{-\int_{0}^{t}{(1,-\omega)\cdot(\mathcal{A}_2-\mc{A}_1)(s,y-s\omega)ds}} |\phi(y)|^2\ dydt\right\vert & \nn\\
		\le C\left(\frac{1}{\sqrt{\lambda}}+ e^{\beta\lambda}\|\Lambda_1-\Lambda_2\|_{*}\right)\|\phi\|_{H^3(\mathbb{R}^n)}^2.\nn
		\end{align}	
	\end{lemma}
	\begin{proof}
		For $u,v\in H^{2}(Q)$ with $u|_{t=0}=\dd_tu|_{t=0}=u|_\Sigma=0$ and $\mc{A}_1,\mc{A}_2\in C_c^{\iy}(Q)^{n+1}$ Green's formula gives	
		\begin{align}
		\int_{Q}(\mathcal{L}_{\mathcal{A}_1,q_1}(u)\overline{v} -u\overline{\mathcal{L}_{\mathcal{-A}_1,\overline{q}_1}v})\ dx dt &= \int_{\Omega}\left(\partial_tu(T,x)\overline{v(T,x)}-u(T,x)\overline{\partial_tv(T,x)}\right)dx\nn\\&\qd  +\int_\Sigma\partial_\nu u(t,x)\overline{v(t,x)}\ dS.\label{Green1}
		\end{align}
		Now we consider geometric optics solutions corresponding to  $\mathcal{L}_{\mathcal{A}_2,q_2}$ and $\mathcal{L}_{-\mathcal{A}_1,\bar{q}_1}$ which we denote by $u_2(t,x)$ and $v(t,x)$ respectively. To cancel the exponential terms, we simultaneously consider exponentially growing and decaying geometric optics solutions. That is, there exists $R^1_{\ld},R^2_{\ld}\in H^2(Q)$	such that for all $\ld\ge\ld_0$
		\begin{align}
		\label{CGO1}&u_2(t,x)=e^{\lambda(t+x\cdot\og)}\left(\phi(x+t\og) e^{-\int_{0}^{t}(1,-\omega)\cdot\mathcal{A}_2(s,x+(t-s)\omega) ds}+R_\lambda^2(t,x)\right),\\
		\label{CGO2}&v(t,x)= e^{-\lambda(t+x\cdot\og)}\left(\phi(x+t\og) e^{\int_{0}^{t}(1,-\omega)\cdot\mathcal{A}_1(s,x+(t-s)\omega) ds}+R_\lambda^1(t,x)\right),\\
		\label{CGO3}&\m{ and }\|R_\lambda^i\|_{H^k(Q)}\le C\lambda^{-1+k}\|{\phi}\|_{H^3(\mathbb{R}^n)};\m{ for }i=1,2\m{ and }k\in\{0,1,2\}.
		\end{align}
		Now taking the initial and boundary data same as $u_2$ we solve the following IBVP and denote its unique solution by $u_1$ i.e.
		\[\begin{cases}
		\mathcal{L}_{\mathcal{A}_1,q_1}u_1=0, \\
		u_1(0,\cdot)=u_2(0,\cdot),\ \dd_tu_1(0,\cdot)=\dd_tu_2(0,\cdot)\ \m{ and }u_1|_{\Sigma}=u_2|_{\Sigma}.	\end{cases}\]
		
		\noindent Let $\mc{A}_i(t,x)\equiv\left(A_{i,k}(t,x)\right)_{0\le k\le n}$ for $i=1,2$ and define $u=u_1-u_2$ in $Q$, then $u$ solves the following problem		
		\begin{align}
		&\mathcal{L}_{\mathcal{A}_1,{q_1}}u(t,x)=(2\mathcal{A}\cdot(\partial_t,-\nabla_x)u_2+\tilde{q}u_2)(t,x),\label{substraction}\\
		&u(0,\cdot)=\partial_tu(0,\cdot)=0\ \m{and } u|_\Sigma=0.\nn
		\end{align}			
		\vspace{-1cm}	
		\begin{align*}\m{where,} \ &\mathcal{A}(t,x)=(\mathcal{A}_2-\mathcal{A}_1)(t,x)\equiv(A_{k}(t,x))_{0\le k\le n}. \\
		&\tilde{q}_i(t,x)=\left(\partial_tA_{i,0}-\sum_{k=1}^n\dd_{x_k} A_{i,k}+{|A_{i,0}|}^2-{\sum_{k=1}^{n}|A_{i,k}|}^2+q_i\right)(t,x),\m{ for }i\in\{1,2\}.\\   &\tilde{q}(t,x)=\left(\tilde{q}_1-\tilde{q}_2\right)(t,x).
		\end{align*}	
		Now we make use of \eqref{Green1} and \eqref{substraction} to get the following integral identity; 
		\begin{align}
		\int_{Q}(2\mathcal{A}\cdot(\partial_t,-\nabla_x)u_2+\tilde{q}u_2)\overline{v}\ dxdt =&\int_{\Omega}\left(\partial_tu(T,x)\overline{v(T,x)}-u(T,x)\overline{\partial_tv(T,x)}\right)dx\nn\\&\  +\int_\Sigma\partial_\nu u(t,x)\overline{v(t,x)}\ dS\label{intidentity}.
		\end{align}
		We wish to modify the integral identity \eqref{intidentity} into an estimate connecting line integrals and input-output operator. We substitute \eqref{CGO1} and \eqref{CGO2} into \eqref{intidentity}. We have
		\begin{align}
		&(\dd_t,-\nabla_x)u_2(t,x)=e^{\lambda(t+x\cdot\omega)}\left(\lambda(1,-\og)\phi(x+t\omega) e^{-\int_{0}^{t}(1,-\omega)\cdot\mathcal{A}_2(s,x+(t-s)\omega) ds}+w_\lambda^2(t,x)\right).\label{derivativeCGO1}\\
		\m{Also, }&\partial_tv(t,x)=e^{-\lambda(t+x\cdot\omega)}\left(\lambda\phi(x+t\omega) e^{\int_{0}^{t}{(1,-\omega)\cdot\mathcal{A}_1(s,x+(t-s)\omega) ds}}+w_\lambda^1(t,x)\right).\label{derivativeCGO2}
		\end{align}
		For $i=1,2$ we see the terms $w^i_{\ld}$ involve derivatives of $\phi$ and $R_{\ld}^i$. For some $C>0$ we arrive at the following estimate using \eqref{CGO3} 
		$$\|w^i_\ld\|_{H^k(Q)}\le C\ld^{k}\|\phi\|_{H^3(\R^n)}\ \ \m{ for }k\in\{0,1\}\m{ and }i\in\{1,2\}.$$ 
		We use \eqref{derivativeCGO1} and \eqref{derivativeCGO2} to get the following relation which will be helpful later as well  
		\begin{align}
		&\mathcal{A}(t,x)\cdot(\partial_t,-\nabla_x)u_2(t,x)\overline{v(t,x)}\nn\\
		&=\left(\ld\mc{A}(t,x)\cdot(1,-\og)|\phi(x+t\omega)|^2 e^{-\int_{0}^{t}(1,-\omega)\cdot\mathcal{A}(s,x+(t-s)\omega) ds}+p_\lambda(t,x)\right)\label{CGO4},\\
		&\m{where, }\ \|p_\ld\|_{L^{1}(Q)}\le C\|{\phi}\|_{H^3(\mathbb{R}^n)}, \  (\m{using a-priori bounds on}\ \mc{A}).\nn
		\end{align}
		Now to estimate R.H.S of \eqref{intidentity} we use explicit bounds for $v$ and the boundary Carleman estimate for $u$ to $\mathcal{L}_{\mathcal{A}_1,q_1}$. We also have
		$$\left(\Lambda_1-\Lambda_2\right)\left(u_2|_{t=0},\dd_t u_2|_{t=0},u_2|_{\Sigma}\right)=\left(\dd_\nu u|_{\Sigma_{-,\epsilon/2}{(\og_0)}},u(T,\cdot)\right).$$
		From \eqref{CGO1} and \eqref{CGO3} and using trace theorem we get $\beta>0$ and $C>0$ such that
		\begin{align}
		&\big\Vert \left(u_2|_{t=0},\dd_t u_2|_{t=0},u_2|_{\Sigma}\right)\big\Vert_{(H^1(\Omega),L^2(\Omega),H^1(\Sigma))}\le Ce^{\beta\ld}\|\phi\|_{H^3(\R^n)}.\label{inputoutput2}\\
		\m{From }\eqref{regularity}\m{ we have, }& \|\dd_\nu u\|_{L^2({\Sigma_{-,\epsilon/2}(\og_0)})},\|u|_{t=T}\|_{H^1(\Omega)}\le Ce^{\beta\ld}\|\Lambda_1-\Lambda_2\|_{*}\|\phi\|_{H^3(\R^n)}.\label{boundaryestimate}
		\end{align}
		Let $\mc{K}$ be the RHS of Boundary Carleman estimate \eqref{CGO1} corresponding to $\mc{L}_{\mc{A}_1,q_1}$ applied on $u$,  
		\begin{align}
		\mc{K}=&\int\limits_{Q}{ e^{-2\lambda(t+x\cdot\omega)}|\mathcal{L}_{\mathcal{A}_1,q_1}u(t,x)|^2dxdt} +\lambda\int_{\Sigma_{-}(\omega)}{ e^{-2\lambda(t+x\cdot\omega)}|\omega\cdot \nu (x)||\partial_\nu u|^2}dS\nn\\
		&\qd+\int\limits_{\Omega} e^{-2\lambda(t+x\cdot\omega)}\left(\ld^2|u(T,x)|^2+\ld|\nabla_x{u(T,x)}|^2\right)\ dx.\nn
		\end{align}	
		We use \eqref{substraction}, \eqref{derivativeCGO1} and a-priori bounds of potentials to get 
		\begin{align}
		\int_{Q}{e^{-2\lambda(t+x\cdot\omega)}|\mathcal{L}_{\mathcal{A}_1,q_1}u|^2dxdt}\le C\lambda^2\|\phi\|_{H^3(\mathbb{R}^n)}^2. 
		\end{align}
		We use continuity of input-output operator defined in \eqref{Defn} and estimates from \eqref{boundaryestimate} to get 
		\begin{align}
		\int\limits_{\Omega}{ e^{-2\lambda(T+x\cdot\omega)}(\ld^2|u(T,x)|^2+\ld|\nabla_x{u(T,x)}|^2}dx)\le C e^{\beta\lambda}\|\Lambda_1-\Lambda_2\|_{*}^2\|\phi\|_{H^3(\mathbb{R}^n)}^2.\label{finaltime}
		\end{align}
		Since $\Sigma_{-}(\omega)\subseteq\Sigma_{-,\ep/2}(\omega_0)$,  using \eqref{boundaryestimate}, we get\\
		\begin{align}
		\int_{\Sigma_{-}(\omega)}{\ld e^{-2\lambda(t+x\cdot\omega)}}|\partial_\nu u|^2 dS\le\int_{\Sigma_{-,\epsilon/2}(\omega_0)}{\ld e^{-2\lambda(t+x\cdot\omega)}}|\partial_\nu u|^2 dS\le C e^{\beta\lambda}\|\Lambda_1-\Lambda_2\|_{*}^2\|\phi\|^2_{H^3(\mathbb{R}^n)}.\label{boundarygreen2}
		\end{align}
		Hence $\mc{K}$ can be bounded by  
		$C (\lambda^2+ e^{\beta\lambda}\|\Lambda_1-\Lambda_2\|_{*}^2 )\|\phi\|_{H^3(\mathbb{R}^n)}^2$. Using it with boundary Carleman estimate we bound each term present in R.H.S of \eqref{intidentity}. 
		We use Holder's inequality and trace theorem to get from \eqref{CGO2}
		\begin{align}
		&\left|\int_{\Omega}\partial_tu(T,x)\overline{v(T,x)}\ dx\right|\nn\\
		&=\left|\int_{\Omega}e^{-\lambda(T+x\cdot\omega)}\partial_tu(T,x)(\phi(x+T\omega){ e}^{-\int_{0}^{T}(1,-\omega)\cdot\mathcal{A}_1(s,x+(T-s)\omega) ds}+R_\lambda^1(T,x))dx\right|,\nn\\
		&\le\f{C\|\phi\|_{H^3(\mathbb{R}^n)}}{\rt\ld}\sqrt{{\int_{\Omega}\ld e^{-2\lambda(T+x\cdot\omega)} 
				|\partial_tu(T,x)|^2}\ dx},\nn\\
		&\le\f{C\|\phi\|_{H^3(\mathbb{R}^n)}}{\rt\ld}\sqrt{\mc{K}}
		\le C\|\phi\|^2_{H^3(\mathbb{R}^n)}\left({\rt\ld} +e^{\beta\lambda}\|\Lambda_1-\Lambda_2\|_{*} \right).\label{boundaryterm}
		\end{align}
		Proceeding similarly we get from \eqref{CGO2} and  \eqref{derivativeCGO2} 
		\begin{align}
		\left|\int_{\Omega}u(T,x)\overline{\partial_tv(T,x)}\ dx\right|&\le C\ld\|\phi\|_{H^3(\mathbb{R}^n)}\sqrt{{\int_{\Omega}e^{-2\lambda(T+x\cdot\omega)}|u(T,x)|^2}\ dx},\nn\\
		&\le Ce^{\beta\lambda}\|\Lambda_1-\Lambda_2\|_{*}\|\phi\|_{H^3(\mathbb{R}^n)}^2 \label{finaltimeestimate}. 
		\end{align}
		\begin{align}
		\m{and, }\left|\int_\Sigma\partial_\nu u(t,x)\overline{v(t,x)}\ dS\right|&=\frac{1}{\sqrt\lambda}\ \left|\int_\Sigma\sqrt{\lambda}\partial_\nu u(t,x)\overline{v(t,x)}\ dS\right|,\nn\\
		&\le \frac{C\|\phi\|_{H^3(\mathbb{R}^n)}}{\sqrt\lambda}\sqrt{\int_{\Sigma}{\lambda e^{-2\lambda(t+x\cdot\omega)}}|\dd_\nu u(t,x)|^2\ dS}.\label{boundarygreen1}
		\end{align}
		We made boundary measurements on more than half of the boundary which is $\Sigma_{-,\epsilon/2}(\omega_0)$. For the part $\Sigma_{+,\epsilon/2}(\omega_0)$ we will be using the boundary Carleman estimate.     
		\begin{align*}
		&\int_{\Sigma}{\lambda e^{-2\lambda(t+x\cdot\omega)}}|\dd_\nu u(t,x)|^2\ dS\\
		&=\int_{\Sigma_{+,\epsilon/2}(\omega_0)}{\lambda e^{-2\lambda(t+x\cdot\omega)}}|\dd_\nu u(t,x)|^2\ dS+\int_{\Sigma_{-,\epsilon/2}(\omega_0)}{\lambda e^{-2\lambda(t+x\cdot\omega)}}|\dd_\nu u(t,x)|^2\ dS.
		\end{align*}
		Since $\Sigma_{+,\ep/2}(\omega_0)\subseteq\Sigma_{+}(\omega)$ we obtain the following
		\begin{align}\int\limits_{\Sigma_{+,\epsilon/2}(\omega_0)}{\ld e^{-2\lambda(t+x\cdot\omega)}}|\partial_\nu u(t,x)|^2\ dS&\le \int\limits_{\Sigma_{+,\epsilon/2}(\omega_0)}\frac{\ld}{\epsilon/2}{ e^{-2\lambda(t+x\cdot\omega)}}|\nu(x)\cdot\omega||\partial_\nu u(t,x)|^2 dS,\nn\\
		&\le\frac{1}{\epsilon/2}\int_{\Sigma_{+}(\omega)}{\lambda e^{-2\lambda(t+x\cdot\omega)}}|\nu(x)\cdot\omega||\partial_\nu u(t,x)|^2 dS,\nn\\
		&\le \f{2}{\epsilon}\mc{K},\qd\qd  \ (\m{using boundary Carleman estimates})\nn\\
		&\le \f{2C}{\epsilon}(\lambda^2+ e^{\beta\lambda}\|\Lambda_1-\Lambda_2\|_{*}^2 )\|\phi\|_{H^3(\mathbb{R}^n)}^2.\label{boundarygreen3}
		\end{align}
		We assemble the estimates \eqref{boundarygreen2},\eqref{finaltimeestimate},\eqref{boundaryterm} and \eqref{boundarygreen3} to obtain from \eqref{intidentity} the following   
		\begin{align}
		\left|\int_{Q}(2\mathcal{A}\cdot(\partial_t,-\nabla_x)u_2+\tilde{q}u_2)(t,x)\overline{v(t,x)}dxdt\right|
		\le C\|\phi\|^2_{H^3(\mathbb{R}^n)}\left({\rt\ld} +e^{\beta\lambda}\|\Lambda_1-\Lambda_2\|_{*} \right).\label{estimate}
		\end{align} 
		We use expressions of geometrical optics solutions which are \eqref{CGO1} and \eqref{CGO2} to get 
		$$\left|\int_{Q}\tilde{q}(t,x)u_2(t,x)v(t,x)\ dxdt\right|\le C\|\phi\|^2_{H^3(\R^n)}.$$
		Thus, dividing \eqref{estimate} by $\ld$ we obtain
		\begin{align*}
		\left\vert\int_{Q}(1,-\omega)\cdot\mathcal{A}(t,x) e^{-\int_{0}^{t}{(1,-\omega)\cdot\mathcal{A}(s,x+(t-s)\og)ds}} |\phi(x+\og t)|^2dxdt\right|\le C\left(\frac{1}{\sqrt{\lambda}}+ e^{\beta\lambda}\|\Lambda_1-\Lambda_2\|_{*}\right)\|\phi\|_{H^3(\R^n)}^2.
		\end{align*}
		Now we make a change of variables which is $y=x+t\og$ and use $\mc{A}\in C_c^{\iy}(Q)^{n+1}$ to obtain 
		\begin{align}
		\left\vert\int_{\mb{R}^n}(1,-\omega)\cdot\mathcal{A}(t,y-t\omega) e^{-\int_{0}^{t}{(1,-\omega)\cdot\mathcal{A}(s,y-s\omega)ds}} |\phi(y)|^2dydt\right|\le C\left(\frac{1}{\sqrt{\lambda}}+ e^{\beta\lambda}\|\Lambda_1-\Lambda_2\|_{*}\right)\|\phi\|_{H^3(\mathbb{R}^n)}^2
		\label{integralidentity}.
		\end{align}
		This completes proof of the lemma.	
	\end{proof}	
	
	To get estimates for light-ray transform from \eqref{integralidentity} we adapt the arguments presented in \cite{BJY2} or \cite{aicha1}. Basically the proof relies on limit passing argument for an approximate identity. \\
	
	\begin{lemma}{3.}
		For all $x\in\mb{R}^n$ and $\og\in\mb{S}^{n-1}$ satisfying $|\og-\og_0|<\f{\epsilon}{2}$, there exist $\delta>0,\ld_0>0,C>0$ such that following holds whenever $\ld\ge\ld_0$
		\begin{align}  
		\left|\int_{\mathbb{R}}{(1,-\omega)\cdot\mathcal{A}(s,x-s\omega)\ ds}\right|\le C\left(\frac{1}{\lambda^\delta}+\rm e^{\beta\lambda}\|\Lambda_1-\Lambda_2\|_{*}\right).\nn
		\end{align}
	\end{lemma}	
	\begin{proof}
		
		We can write
		\begin{align}   
		&\int_{0}^{T}\int_{\mathbb{R}^n}(1,-\omega)\cdot\mathcal{A}(t,y-t\omega) e^{-\int_{0}^{t}(1,-\omega)\cdot\mathcal{A}(s,y-s\omega)ds} |\phi(y)|^2\ dydt\nn\\
		&=-\int_{\mathbb{R}^n}\int_{0}^{T}|\phi(y)|^2\partial_t{\left( e^{-\int_{0}^{t}{(1,-\omega)\cdot\mathcal{A}(s,y-s\omega)ds}}\right)}\ dtdy,\nn\\
		\label{f5}&=-\int_{\mathbb{R}^n}|\phi(y)|^2\left( e^{-\int_{0}^{T}{(1,-\omega)\cdot\mathcal{A}(s,y-s\omega)ds}}-1\right)dy.
		\end{align}
		Fix $x\in\mb{R}^n$ and choose $\phi\in C_c^{\iy}(B(0,1))$ with $\|\phi\|_{L^2(\mathbb{R}^n)}=1$ where $B(0,1)$ is the open unit ball in $\R^n$.
		Let us define 
		\[ \phi_h(y)=h^{-\frac{n}{2}}\phi(\frac{y-x}{h})\m{ in }\mathbb{R}^n\m{   for   }h>0.\] 
		Then $\phi_h\in C_c^{\iy}(\R^n)$ and there exists $C>0$ which depends on $\phi$ such that following holds 
		\[supp(\phi_h)\subseteq B(x,h),\|\phi_h\|_{L^2(\mathbb{R}^n)}=1 \m{ and } \|\phi_h\|_{H^3(\mathbb{R}^n)}\le Ch^{-3}.\] 
		Now using \eqref{f5} we get from \eqref{integralidentity}
		\begin{equation}
		\left|\int_{\R^n}|\phi_h(y)|^2(e^{-\int_{0}^{T}(1,-\og)\cdot\mc{A}(s,y-s\og)ds}-1)dy\right|\le Ch^{-3}\left(\frac{1}{\sqrt{\lambda}}+e^{\beta\lambda}\|\Lambda_1-\Lambda_2\|_{*}\right).\label{approximation1}
		\end{equation}
		\vspace*{1mm}  
		As $\mathcal{A}\in C_c^{\iy}(Q)^{n+1}$, using mean value theorem twice we get $C>0$ which depends on $Q$ and a-priori bounds of $\mc{A}_i$ such that the following holds
		\begin{align}
		&\left|e^{-\int_{0}^{T}{(1,-\omega)\cdot\mathcal{A}(s,x-s\omega)ds}}- e^{-\int_{0}^{T}{(1,-\omega)\cdot\mathcal{A}(s,y-s\omega)ds}}\right|\nn\\
		&\le C\left|\int_{0}^{T}{(1,-\omega)\cdot\mathcal{A}(s,x-s\omega)ds}-\int_{0}^{T}{(1,-\omega)\cdot\mathcal{A}(s,y-s\omega)ds}\right|\le C|x-y|.\label{approximation2}
		\end{align}
		Consider the following positive continuous function on $\mb{R}$  
		\begin{align}
		f(x)=\begin {cases}
		\f{e^x-1}{x}\qd &\m{for }x\neq 0,\\
		1\qd\qd&\m{for }x=0.
		\end{cases}\label{func}
		\end{align}
		
		\noindent For $M>0$ we use continuity of $f$ on the compact interval $[-M,M]$ to get $C>0$ which depends on $M$ such that 
		\[|x|\le C|e^x-1| \qd\qd\qd\fa x\in[-M,M]. \]
		Since $\mc{A}\in C_c^{\iy}(Q)^{n+1}$ we get $C>0$ depending on $Q$ and a-priori bounds of $\mc{A}_i$
		such that the following estimate holds
		\[\left|\int_{0}^{T}(1,-\og)\cdot\mc{A}(s,x-s\og)ds\right|\le C\left| e^{-\int_{0}^{T}(1,-\og)\cdot\mc{A}(s,x-s\og)ds}-1\right|.\]
		Now, \begin{align}
		&\left|e^{-\int_{0}^{T}(1,-\og)\cdot\mc{A}(s,x-s\og)ds}-1\right|=\left|\int_{\R^n}|\phi_h(y)|^2(e^{-\int_{0}^{T}-(1,-\og)\cdot\mc{A}(s,x-s\og)ds}-1)\ dy\right|,\nn\\
		&\le\  \left|\int_{\R^n}|\phi_h(y)|^2(e^{-\int_{0}^{T}(1,-\og)\cdot\mc{A}(s,x-s\og)ds}-e^{-\int_{0}^{T}(1,-\og)\cdot\mc{A}(s,y-s\og)ds})\ dy\right|\nn\\
		&\qd+\left|\int_{\R^n}|\phi_h(y)|^2(e^{-\int_{0}^{T}(1,-\og)\cdot\mc{A}(s,y-s\og)ds}-1)\ dy\right|,\m{ (using triangle inequality)}\nn\\
		&\le  C\left(h+h^{-3}\left(\frac{1}{\sqrt{\lambda}}+ e^{\beta\lambda}\|\Lambda_1-\Lambda_2\|_{*}\right)\right).\ \ (\m{using }\eqref{approximation2}\m{ and }\eqref{approximation1}) \label{lineestimate1}
		\end{align}    
		We choose small enough $h>0$ such that $h$ and $\frac{h^{-3}}{\sqrt{\lambda}}$ are comparable. It can be done by taking $h=\ld^{-\f{1}{8}}$. Then for $\lambda\ge\lambda_0$ and  $|\og-\og_0|<\f{\epsilon}{2}$ we get from $\eqref{lineestimate1}$ there exists $\delta>0$ such that
		\begin{align}
		\left|\int_{\mathbb{R}}{(1,-\omega)\cdot\mathcal{A}(s,x-s\omega)ds}\right|&=\left|\int_{0}^{T}{(1,-\omega)\cdot\mathcal{A}(s,x-s\omega)ds}\right|,\nn\\
		\label{line integral estimate}&\le C\left(\frac{1}{\lambda^\delta}+ e^{\beta\lambda}\|\Lambda_1-\Lambda_2\|_{*}\right).
		\end{align} 
		This completes the proof.
	\end{proof}
	
	Now we use ideas from \cite{BR} to have Fourier estimates of some specific components of vector potentials along right rays. 
	
	\begin{corollary}{2.}
		There exist an open cone $\mc{C}$ in $\mb{R}^{n+1}$ such that for $(\tau,\xi)\in\mc{C}$ the following holds for all $\og(\tau,\xi)\in\mb{S}^{n-1}$ satisfying $|\og-\og_0|<\f{\epsilon}{2}$ and $\og(\tau,\xi)\cdot\xi=\tau$, \[|(1,-\og(\tau,\xi))\cdot\widehat{\mc{A}}(\tau,\xi)|\le C\left(\frac{1}{\lambda^\delta}+ e^{\beta\lambda}\|\Lambda_1-\Lambda_2\|_{*}\right).
		\]	
	\end{corollary}  
	
	\begin{proof}
		Consider $x\in\R^n$ and $\og\in\mb{S}^{n-1}$ satisfying $|\og-\og_0|<\ep/2$. Then we have
		\begin{align}
		&|(1,-\og)\cdot\widehat{\mc{A}}(\og\cdot\xi,\xi)|=\left|\int_{\R^{1+n}}e^{-iy\cdot \xi}e^{-i(\og\cdot\xi)s}(1,-\og)\cdot\mc{A}(s,y)\ dsdy\right|,\nn\\
		&\overset{(y\to z-s\og)}{=}\left|\int_{\R^n}e^{-iz\cdot\xi}\left(\int_{\R}(1,-\og)\cdot {\mc{A}}(s,z-\og s)ds\right)\ dz\right|,\nn\\
		&\qd\le\int_{\R^n}\left|\int_{\R}(1,-\og)\cdot\mc{A}(s,z-s\og)ds\right|\ dz,\nn\\
		\label{C1}&\qd\le C\left(\frac{1}{\lambda^\delta}+e^{\beta\lambda}\|\Lambda_1-\Lambda_2\|_{*}\right).\qd (\m{ using }\eqref{line integral estimate}\m{ and }\mc{A}\in C_c^{\iy}(Q)^{n+1}) 
		\end{align}
		Now as in \cite{BR} we characterize points of $\R^{n+1}$ for which estimate $\eqref{C1}$ holds.
		So we define \[K_{\epsilon}=\cup_{|\og-\og_0|<{\epsilon/2}}\og^\perp\m{ and } E_{\epsilon}=\{(\tau,\xi)\in\R\times K_{\epsilon/2};|\tau|<\f{\epsilon}{8}|\xi|\}.\] 
		Since interior of $K_\epsilon$ being nonempty, $E_\epsilon$ contains an open cone say $\mc{C}$.
		Now we will show that if $(\tau_0,\xi_0)\in E_{\epsilon}$, then there exists $\og\in\mb{S}^{n-1}$ such that $|\og-\og_0|<\f{\epsilon}{2}$ and $t_0=\og\cdot x_0$.\\\\
		Assume $(\tau_0,\xi_0)\in E_\epsilon$ then $\xi_0\cdot\og_1=0$ for some $\og_1\in\mb{S}^{n-1}$ with $|\og_1-\og_0|<\f{\epsilon}{4}$. Now we take,  $\og=\f{\tau_0}{|\xi_0|^2}\xi_0+\rt{1-\f{\tau_0^2}{|\xi_0|^2}}\og_1.$ We see then $\tau_0=\xi_0\cdot \og.$\\ 
		
		\noindent By our choice of $\og_0\in\mb{S}^{n-1}$ we observe
		\begin{align}
		|\og-\og_0|&\le\left|\f{\tau_0}{|\xi_0|^2}\right||\xi_0|+\rt{1-\f{\tau_0^2}{|\xi_0|^2}}|\og_1-\og_0|+\left|\left(\rt{1-\f{\tau_0^2}{|\xi_0|^2}}-1\right)\og_0\right|\nn\\
		&\le\f{\epsilon}{8}+\f{\epsilon}{4}+\f{\epsilon}{8}=\f{\epsilon}{2}.\nn
		\end{align}
		Thus for $(\tau,\xi)\in E_\epsilon$ we obtain, $ |(1,-\og)\cdot\widehat{\mc{A}}(\tau,\xi)|\le C(\frac{1}{\lambda^\delta}+ e^{\beta\lambda}\|\Lambda_1-\Lambda_2\|_{*})$  where  $\og\in\mb{S}^{n-1}$ with $|\og-\og_0|<\epsilon/2$ so that $\tau=\og\cdot\xi$ .
	\end{proof}
	
	Now we will obtain some uniform norm estimate of vector potentials over a cone so that we can take advantage of Vessella's analytic continuation argument for estimating Fourier transform of vector potentials over large balls. We follow the arguments of \cite{salazar 2} in this regard, but we first consider the full data case to expound Lemma 2.5 of \cite{salazar 2}. Then we present the partial data case providing with an explanation why the method does not work for $n=2$.

	\begin{lemma}{4.} 
		For fixed $(\tau,\xi)\in\mb{S}^{n}$ satisfying $|\tau|< \f{1}{2}|\xi|$ consider the following set of equations
		\begin{align}
		(1,-\omega(\tau,\xi))\cdot\widehat{\mc{A}}(\tau,\xi)&= G(\xi,\omega(\tau,\xi)),\label{ray transform}\\
		(\tau,\xi)\cdot\widehat{\mc{A}}(\tau,\xi)&=0,\qd(\m{This is because }\mc{A} \m{ is divergence-free})\label{divergence free}\\
		\m{ where } \omega(\tau,\xi)&\in\mb{S}^{n-1} \m{ satisfies } \omega(\tau,\xi)\cdot\xi=\tau.\label{unit vector}
		\end{align} 
		
		\noindent Then there exist $C>0$ independent of $\{(\tau,\xi)\in\mb{S}^{n};|\tau|<\f{1}{2}|\xi|\}$ such that for some choice of $\omega^k(\tau,\xi)$'s satisfying \eqref{unit vector} we have 
		
		\[ |\widehat{A}_j(\tau,\xi)|\le C\max\limits_{1\le k\le n}|G(\xi,\omega^k(\tau,\xi))|, \qd \forall j\in\{0,1,...,n\}.\]
	\end{lemma}
	
	\begin{proof} In order to make the presentation clear we assume $\xi=e_n$, otherwise one can consider appropriate rotation in $\mb{R}^{n}$ which does not change the end result as orthogonal transformations respect inner product. Thus the unit vectors $\omega$ in \eqref{ray transform} satisfy $\omega_n=\tau$. We observe then 
		
		\vspace*{-5mm}
		
		\begin{align*}
		\sum_{k=1}^{n-1}\omega_k^2=1-\omega_n^2=1-\tau^2 >\f{3}{4}.
		\end{align*}
		
		\noindent Hence all those $\omega$'s can be parametrized by $r\mb{S}^{n-2}$ where $\f{\sqrt{3}}{2}<r\le 1$. Let us choose $n-1$ orthogonal vectors from $r\mb{S}^{n-2}$ and denote them by $\tilde{\omega}^i$ for $i\in\{1,2...,n-1\}$. We define 
		
		\[\tilde{\omega}^n=\f{1}{\sqrt{2}}(\tilde{\omega}^{n-2}+\tilde{\omega}^{n-1}).\]
		
		\noindent Now we will be considering \eqref{ray transform} for ${\omega}_i$'s where ${\omega}^i=(\tilde{\omega}^i,\tau)$ for $i\in\{1,2...,n\}$. So the system of equations we are interested is the following
		
		\begin{align*}
		\widehat{A_0}(\tau,\xi)-\sum_{j=1}^{n}{\omega}^i_j\widehat{A}_j(\tau,\xi)&=G(\xi,\omega^i(\tau,\xi)), \qd i\in\{1,2,...n\}.\\
		\f{1}{\sqrt{\tau^2+|\xi|^2}}(\tau\widehat{A_0}(\tau,\xi)+\sum_{j=1}^{n}\xi_j\widehat{A}_j(\tau,\xi))&=0.
		\end{align*} 
		
		\noindent Unique solvability of the above system follows from the fact that we are assuming divergence free potential and orthogonal complement of $\{(1,-\omega(\tau,\xi));\omega(\tau,\xi)\in\mb{S}^{n-1}\mbox{ and }\tau+\xi\cdot\omega(\tau,\xi)=0\}$ is one dimensional (see Appendix of \cite{salazar}). For stable recovery of the potentials we want to obtain a positive lower bound on the absolute value of determinant of the matrix $M(\tau,\xi)$ defined by 
		
		\begin{align}
		M_{(\tau,\xi)}=\begin{pmatrix}
		1 & -\og_1^1(\tau,\xi) & \cdots & -\og_1^n(\tau,\xi) \\
		1 & -\og_2^1(\tau,\xi) & \cdots & -\og_2^n(\tau,\xi) \\
		\vdots  & \vdots  & \ddots & \vdots  \\
		1 & -\og_n^1(\tau,\xi) & \cdots & -\og_n^n(\tau,\xi) \\
		\f{\tau}{\rt{{\tau}^2+|\xi|^2}} & \f{\xi_1}{\rt{{\tau}^2+|\xi|^2}} & \cdots & \f{\xi_n}{\rt{{\tau}^2+|\xi|^2}} 
		\end{pmatrix}.\label{matrix} 	
		\end{align}
		
		\noindent If $V(\tau,\xi)$ is the $n$ dimensional volume generated by the vectors $\{(1,-\omega^i(\tau,\xi))\}_{1\le i\le n}$, then
		
		\[\mbox{det}M(\tau,\xi)=V(\tau,\xi)\times P(\tau,\xi).\]
		
		\noindent Here $P(\tau,\xi)$ is length of the component of $(\tau,\xi)$ which is perpendicular to the subspace generated by $\{(1,-\omega^i(\tau,\xi))\}_{1\le i\le n}$. Since $(\tau,\xi)$ is at least $\f{\pi}{8}$ angle away from the light cone and the vectors $\{(1,-\omega^i(\tau,\xi))\}_{1\le i\le n}$ lie on the boundary of light cone we have  \[|V(\tau,\xi)|\sin\left(\f{\pi}{8}\right)\le |\mbox{det}M(\tau,\xi)|.\]    
		
		\noindent To compute $V(\tau,\xi)$ we consider the Gramian of $\{(1,-\omega^i(\tau,\xi))\}_{1\le i\le n}$ denoted by $\mc{G}(\tau,\xi)$. For convenience we denote the unit vectors satisfying \eqref{unit vector} as $\omega$ only. We see
		
		\begin{align}
		\mc{G}(\tau,\xi)=&
		\begin{vmatrix}
		(1,-\omega^1)\cdot(1,-\omega^1) & (1,-\omega^1)\cdot(1,-\omega^2)& \dots & (1,-\omega^1)\cdot(1,-\omega^n)\\
		(1,-\omega^2)\cdot(1,-\omega^1) & (1,-\omega^2)\cdot(1,-\omega^2)& \dots & (1,-\omega^n)\cdot(1,-\omega^2)\\
		\vdots  & \vdots       & \ddots & \vdots\\
		(1,-\omega^n)\cdot(1,-\omega^1) & (1,-\omega^n)\cdot(1,-\omega^2)& \dots & (1,-\omega^n)\cdot(1,-\omega^n)
		\end{vmatrix}\nn\\
		=&
		\begin{vmatrix}
		1+\tau^2+\|\tilde{\omega}^1\|^2 & 1+\tau^2+\tilde{\omega}^1\cdot\tilde{\omega}^2& \dots & 1+\tau^2+\tilde{\omega}^1\cdot\tilde{\omega}^n\\
		1+\tau^2+\tilde{\omega}^2\cdot\tilde{\omega}^1 & 1+\tau^2+\|\tilde{\omega}^2\|^2& \dots & 1+\tau^2+\tilde{\omega}^2\cdot\tilde{\omega}^n\\
		\vdots  & \vdots       & \ddots & \vdots\\
		1+\tau^2+\tilde{\omega}^n\cdot\tilde{\omega}^1 & 1+\tau^2+\tilde{\omega}^n\cdot\tilde{\omega}^2& \dots & 1+\tau^2+\|\tilde{\omega}^n\|^2\label{gramian}
		\end{vmatrix}.
		\end{align} 
		
		\noindent We now use multi-linearity property of determinants to expand \eqref{gramian} and the fact that Gramian of $\{\tilde{\omega}^i\}_{1\le i\le n}$ is zero as they were chosen from $\mb{R}^{n-1}$. That is	
		\begin{align}
		\mc{G}(\tau,\xi)=\sum_{k=1}^{n}(1+\tau^2)\ B_k(\tau,\xi),\label{determinant sum}
		\end{align}
		
		where $B_k(\tau,\xi)$ has $k^{th}$ column as $(1,1,....,1)^{t}$ and $j^{th}$ as $(\tilde{\omega}^1\cdot\tilde{\omega}^j,\tilde{\omega}^2\cdot\tilde{\omega}^j,....,\tilde{\omega}^n\cdot\tilde{\omega}^j)^{t}$ for $j\neq k$.
		
		\vspace*{2mm}
		
		\noindent We chose $\tilde{\omega}^n$ as a linear combination of $\tilde{\omega}^{n-2}$ and $\tilde{\omega}^{n-1}$. Thus $B_k(\tau,\xi)=0$ for $1\le k\le n-3$. To calculate rest of the terms in \eqref{determinant sum} we observe  
		
		\begin{align}
		B_{n-2}(\tau,\xi)=&
		\begin{vmatrix}
		\|\tilde{\omega^1}\|^2 & \dots & 1 & \tilde{\omega}^1\cdot\tilde{\omega}^{n-1} & \tilde{\omega}^1\cdot\tilde{\omega}^{n} \\
		\tilde{\omega}^2\cdot\tilde{\omega}^1  &\dots & 1 & \tilde{\omega}^2\cdot\tilde{\omega}^{n-1} & \tilde{\omega}^2\cdot\tilde{\omega}^n \\
		\vdots  & \vdots  & \vdots & \vdots & \vdots \\
		\tilde{\omega}^n\cdot\tilde{\omega}^1 & \dots & 1 & \tilde{\omega}^n\cdot\tilde{\omega}^{n-1} & \|\tilde{\omega}^n\|^2
		\end{vmatrix},\nn   \\
		=&\f{1}{\sqrt{2}}\begin{vmatrix}
		\|\tilde{\omega^1}\|^2 & \dots & 1 & \tilde{\omega}^1\cdot\tilde{\omega}^{n-1} & \tilde{\omega}^1\cdot(\tilde{\omega}^{n-2}+\tilde{\omega}^{n-1}) \\
		\tilde{\omega}^2\cdot\tilde{\omega}^1  &\dots & 1 & \tilde{\omega}^2\cdot\tilde{\omega}^{n-1} & \tilde{\omega}^2\cdot(\tilde{\omega}^{n-2}+\tilde{\omega}^{n-1}) \\
		\vdots  & \vdots  & \vdots & \vdots & \vdots \\
		\tilde{\omega}^n\cdot\tilde{\omega}^1 & \dots & 1 & \tilde{\omega}^n\cdot\tilde{\omega}^{n-1} & \omega^{n}\cdot(\tilde{\omega}^{n-2}+\tilde{\omega}^{n-1})
		\end{vmatrix},\nn
		\end{align}
		
		\begin{align}     
		=&\f{1}{\sqrt{2}}
		\begin{vmatrix}
		\|\tilde{\omega^1}\|^2 & \dots & 1 & \tilde{\omega}^1\cdot\tilde{\omega}^{n-1} & \tilde{\omega}^1\cdot\tilde{\omega}^{n-2} \\
		\tilde{\omega}^2\cdot\tilde{\omega}^1  &\dots & 1 & \tilde{\omega}^2\cdot\tilde{\omega}^{n-1} & \tilde{\omega}^2\cdot\tilde{\omega}^{n-2} \\
		\vdots  & \vdots  & \vdots & \vdots & \vdots \\
		\tilde{\omega}^n\cdot\tilde{\omega}^1 & \dots & 1 & \tilde{\omega}^n\cdot\tilde{\omega}^{n-1} & \tilde{\omega}^n\cdot\tilde{\omega}^{n-2}
		\end{vmatrix},\nn\\
		=&-\f{1}{\sqrt{2}}B_{n}(\tau,\xi).\nn
		\end{align} 
		
		\noindent Similarly we can have $B_{n-1}(\tau,\xi)=-\f{1}{\sqrt{2}}B_{n}(\tau,\xi).$ Consequently we get  
		\begin{align}
		\mc{G}(\tau,\xi)=(1+\tau^2)(1-\sqrt{2})B_{n}(\tau,\xi).\label{determinant 1}
		\end{align}
		
		\noindent Now,
		\begin{align}
		B_{n}(\tau,\xi)=&
		\begin{vmatrix}
		\|\tilde{\omega}\|^2 & \dots & \tilde{\omega}^1\cdot\tilde{\omega}^{n-2} & \tilde{\omega}^1\cdot\tilde{\omega}^{n-1} & 1 \\
		\tilde{\omega}^2\cdot\tilde{\omega}^1  &\dots & \tilde{\omega}^2\cdot\tilde{\omega}^{n-2} & \tilde{\omega}^2\cdot\tilde{\omega}^{n-1} & 1 \\
		\vdots  & \vdots  & \vdots & \vdots & \vdots \\
		\tilde{\omega}^n\cdot\tilde{\omega}^1 & \dots & \tilde{\omega}^n\cdot\tilde{\omega}^{n-2} & \tilde{\omega}^n\cdot\tilde{\omega}^{n-1} & 1
		\end{vmatrix}
		= \begin{vmatrix}
		\|\tilde{\omega^1}\|^2 & \dots & \tilde{\omega}^1\cdot\tilde{\omega}^{n-3} \\
		\tilde{\omega}^2\cdot\tilde{\omega}^1  &\dots & \tilde{\omega}^2\cdot\tilde{\omega}^{n-3} \\
		\vdots  & \vdots & \vdots\\
		\tilde{\omega}^{n-3}\cdot\tilde{\omega}^1 & \dots & \|\tilde{\omega}^{n-3}\|^2
		\end{vmatrix}
		\begin{vmatrix}
		r^2 & 0 & 1\\
		0 & r^2 & 1\\
		\f{r^2}{\sqrt 2} & \f{r^2}{\sqrt 2} & 1
		\end{vmatrix}\nn\\
		=&r^{2(n-1)}(1-\sqrt 2).\label{determinant 2}
		\end{align}
		
		\noindent We combine \eqref{determinant 1} and \eqref{determinant 2} to conclude $0<c\le|V(\tau,\xi)| $ where $c$ is independent of $(\tau,\xi)$. 
		By cofactor expansion of the matrix \eqref{matrix} we solve the linear system [\ref{ray transform}, \ref{divergence free}] as 
		
		\vspace{-5mm}
		
		\begin{align}
		\label{cofactor expansion 1}\widehat{A}_j(\tau,\xi)=\sum_{k=1}^{n}c_{k,j}(\tau,\xi)G(\xi,\omega^k(\tau,\xi)),\qd j\in\{0,1,...,n\},
		\end{align}
		
		\vspace{-3mm}
		
		where $c_{k,j}(\tau,\xi)=\f{1}{\mbox{det} M(\tau,\xi)}C_{j,k}(\tau,\xi)$ and $C_{j,k}(\tau,\xi)$ is $(j,k)^{th}$ cofactor of $M(\tau,\xi)$. 
		
		\vspace{2mm}
		
		\noindent All entries of $M(\tau,\xi)$ have absolute value less than or equal to one and $C_{j,k}(\tau,\xi)$ consist of products of these terms. So there exist $C>0$ independent of $\{(\tau,\xi)\in\mb{S}^n;|\tau|<\f{1}{2}|\xi|\}$ such that the following holds
		
		\[ |\widehat{A}_j(\tau,\xi)|\le C\max\limits_{1\le k\le n}|G(\xi,\omega^k(\tau,\xi))|,\qd \fa j\in\{0,1,..,n\}.\]
		
	\end{proof} 
	
	\vspace*{-5mm}
	
	We now make a modification in the arguments presented above to invert the light-ray transform for partial data case. Before that we make the following observation. 
	
	\begin{lemma}{5.}
		For fixed $(\tau,\xi)\in\mb{R}^{n+1}$ satisfying $|\tau|<\f{\epsilon}{8}|\xi|$,  the following set if non-empty has atleast two linearly independent vectors when $n\ge3$.
		\[\mc{B}_{\tau,\xi}=\{\omega\in\mb{S}^{n+1};\omega\cdot\xi=\tau \mbox{ and }|\omega-\omega_0|<\epsilon\}. \]
		
		\begin{proof}
			Consider the following map defined on the spherical cap $\{\omega\in\mb{S}^{n-1};|\omega-\omega_0|<\epsilon\}$  
			\[f(\omega)=\omega\cdot\xi-\tau.\]
			
			\noindent We assume the zero-set of continuous function $f$ is non-empty and wish to show it contains atleast two linearly independent vectors. If possible let it have only one element say $\bar\omega$. We observe then 
			
			\vspace{-4mm}
			
			\begin{align}
			\{\omega\in\mb{S}^{n-1};\omega\neq\bar\omega\mbox{ and }|\omega-\omega_0|<\epsilon\}=f^{-1}(-\infty,0)\cup f^{-1}(0,\infty).\label{zero set}  
			\end{align}
			
			\noindent For $n\ge3$ we notice the set in the L.H.S of \eqref{zero set} is connected whereas image of it under $f$ is not. Hence $\mc{B}_{\tau,\xi}$ must have two linearly independent vectors. This argument does not hold for the case $n=2$, as punctured spherical caps are not connected there. For $n=2$, we see $\mc{B}_{\tau,\xi}$ has atmost one vector when $\ep>0$ is small enough.
			
		\end{proof}
	\end{lemma}  
	
	\vspace{-0.7cm}
	
	\noindent Now for $n\ge 3$, we show $\mc{B}_{\tau,\xi}$ has enough linearly independent vectors so that one can invert the matrix $M_{(\tau,\xi)}$.  
	
	\begin{lemma}{6.}   
		For $n\ge 3$, $\mc{B}_{\tau,\xi}$ has $n$ linearly independent vectors.
	\end{lemma}  
	
	\begin{proof}
		Without loss of generality we may take $\xi=e_n$, otherwise one consider $x_n$-axis along the vector $\xi$. Our representation of the unit vector $\omega$ in spherical coordinates is as follows
		
		\vspace*{-4mm}
		
		\begin{align*}
		&\omega_1=\sin\theta_1\sin\theta_2....\sin\theta_n,\\
		&\omega_2=\sin\theta_1\sin\theta_2....\cos\theta_n,\\
		&\vdots\\
		&\og_n=\cos\theta_1.
		\end{align*}
		
		\noindent For $\bar{\omega}\in\mc{B}_{\tau,\xi}$ we represent the angles in bar and we have then  $\cos\bar{\theta}_1\in\left(0,\f{\epsilon}{8}\right)$. We want to show that $\mc{B}_{\tau,\xi}$ contains $n$ linearly independent vectors. If not then $\mc{B}_{\tau,\xi}$ lie on a plane in $\mb{R}^{n+1}$, hence we should have a non-zero vector say $\alpha\in\mb{R}^{n}$ such that for $(\theta_2,\dots,\theta_n)$ varying in a neighborhood of $(\bar\theta_2,\dots,\bar\theta_n)$ we have  
		\begin{align}
		\alpha\cdot(\sin\bar\theta_1\sin\theta_2....\sin\theta_n,\sin\bar\theta_1\sin\theta_2....\cos\theta_n,\dots,\sin\bar\theta_1\cos\theta_2,\cos\bar\theta_1)=0\label{perpendicular}
		\end{align}
		We now differentiate \eqref{perpendicular} with respect to $\theta_2$ twice to get the following 
		
		\vspace*{-4mm}
		
		\begin{align}
		-\alpha\cdot(\sin\bar\theta_1\sin\theta_2....\sin\theta_n,\sin\bar\theta_1\sin\theta_2....\cos\theta_n,\dots,\sin\bar\theta_1\cos\theta_2,0)=0\label{derivative theta 2}
		\end{align}
		
		\noindent We now combine \eqref{perpendicular} and \eqref{derivative theta 2} to conclude $\alpha_n=0$ as, $\cos\bar{\theta}_1\in\left(0,\f{\epsilon}{8}\right)$. Now we differentiate \eqref{perpendicular} with respect to $\theta_3$ twice to get 
		
		\vspace*{-4mm}
		
		\begin{align}
		-\alpha\cdot(\sin\bar\theta_1\sin\theta_2....\sin\theta_n,\sin\bar\theta_1\sin\theta_2....\cos\theta_n,\dots,\sin\bar\theta_1\sin\theta_2\cos\theta_3,0,0)=0\label{derivative theta 3}
		\end{align}
		
		\noindent We add \eqref{derivative theta 3} and \eqref{perpendicular} to get $\alpha_{n-1}\sin\bar\theta_1\cos\theta_2=0$ which in turn implies $\alpha_{n-1}\sin\bar\theta_1=0$ as $\theta_2$ was varying in a neighborhood of $\bar\theta_2$. This yields $\alpha_{n-1}$ is actually zero. Proceeding similarly we can show $\alpha=0$ which contradicts our assumption on $\alpha$. 
	\end{proof}
	
	\begin{lemma} {7.} 
		Let $\mc{C}$ be the open cone as described in Corollary 2. Then there exists an open cone $\mc{C}_0\subseteq\mc{C}$ such that the following holds
		\[|\widehat{\mc{A}}(\tau,\xi)|\le C\left(\frac{1}{\lambda^\delta}+ e^{\beta\lambda}\|\Lambda_1-\Lambda_2\|_{*}\right),\qd \m{for }(\tau,\xi)\in\mc{C}_0.\]
	\end{lemma}
	
	\begin{proof}
		Let $(\tau_0,\xi_0)\in \mc{C}\cap\mb{S}^n$ then from Lemma we have a set of $n$ linearly independent vectors from $\mc{B}_{\tau,\xi}$ denoted by $\{\omega_k(\tau_0,\xi_0)\}_{1\le k\le n}$ which implies that the vectors $\{(1,-\omega_k(\tau,\xi))\}_{1\le k\le n}$ are linearly independent too. As $(\tau_0,\xi_0)$ is perpendicular to each of those vectors we get the matrix $M_{(\tau_0,\xi_0)}$ is invertible. Now as we change $(\tau,\xi)$ continuously in a neighborhood of $(\tau_0,\xi_0)$ in $\mb{S}^n$ say $\tilde{\mc{C}}$, we notice the hyperplane $\xi\cdot x=\tau$ moves in a continuous way. Hence we get $n$ linearly independent vectors from $\mc{B}_{\tau,\xi}$ denoted as $\{\omega_k(\tau,\xi)\}_{1\le k\le n}$ depending continuously on $(\tau,\xi)$. As before $M_{(\tau,\xi)}$ becomes invertible for $(\tau,\xi)\in\tilde{\mc{C}}$. For a compactly contained open set in $\tilde{\mc{C}}$ say $\tilde{\tilde{\mc{C}}}$ we have 
		$0<c\le|det M_{(\tau,\xi)}|$. Now for $(\tau,\xi)\in\tilde{\tilde{\mc{C}}}$ we consider the system of equations 
		
		\begin{align}
		\label{system1}(1,-\og_k(\tau,\xi))\cdot \widehat{\mc{A}}(\tau,\xi)&=G(\xi,\omega_k(\tau,\xi)),\qd\m{ for }k\in\{1,2,...n\},\\
		\label{system2}\f{1}{\sqrt{\tau^2+\xi^2}}(\tau,\xi)\cdot \widehat{\mc{A}}(\tau,\xi)&=0, \ \ (\m{This is because }\mc{A}\m{ is divergence-free}).
		\end{align}
		
		\noindent We observe that Corollary 2 gives an upper bound of $G$ in the cone $\mc{C}$ and $M_{(\tau,\xi)}$ is a non singular homogeneous matrix of degree zero. We use these facts to obtain estimates of the vector potentials in an open cone in $\mb{R}^{n+1}$ from the estimates over $\tilde{\tilde{\mc{C}}}$. We proceed as in \eqref{cofactor expansion 1} to get from \eqref{system1} and \eqref{system2} for $(\tau,\xi)\in\tilde{\tilde{\mc{C}}}$ and $r>0$  	
		\begin{align}
		\label{cofactor expansion 2}\widehat{A}_j(r\tau,r\xi)&=\sum_{k=1}^{n}c_{k,j}(r\tau,r\xi)G(r\xi,\omega^k(r\tau,r\xi)),\qd j\in\{0,1,...,n\}\\
		&=\sum_{k=1}^{n}c_{k,j}(\tau,\xi)G(r\xi,\omega^k(\tau,\xi)).
		\end{align}
		where $c_{k,j}(\tau,\xi)=\f{1}{\mbox{det} M(\tau,\xi)}C_{j,k}(\tau,\xi)$ and $C_{j,k}(\tau,\xi)$ is $(j,k)^{th}$ cofactor of $M(\tau,\xi)$. Hence for the open cone $\mc{C}_0\ (\equiv\cup_{r>0}r\tilde{\tilde{\mc{C}}})$ in $\mb{R}^{n+1}$ we use Corollary to obtain $C>0$	such that 
		\begin{align}
		|\widehat{A}_j(\tau,\xi)|\le C\left(\frac{1}{\lambda^\delta}+ e^{\beta\lambda}\|\Lambda_1-\Lambda_2\|_{*}\right)\qd\m{ for }j\in\{0,1,...,n\}.\label{oncone}
		\end{align}
	\end{proof}
   
     \vspace{-0.7cm}
     \noindent \tb{{Remark}:} We make a remark for the specific case when there is no time-derivative perturbation present that is $A_0(t,x)=0$ in $Q$. Then we can get the exactly same Fourier estimates over an open cone as obtained in \eqref{oncone} for the rest of the components of the vector potential without the divergence-free assumption. This is because for a fixed $(\tau,\xi)\in\tilde{\tilde{\mc{C}}}$ we have $n$ linearly independent vectors $\{\og_k(\tau,\xi)\}_{1\le k\le n}$ as before which makes \eqref{system1} an invertible system consisting of $n$ unknowns to be determined from $n$ equations. After that one carries out the arguments presented below to arrive at stability results similar to \eqref{main result 1} and \eqref{main result 2}. Thus we do not need any divergence-free condition (with respect to space variables) on the vector potentials to obtain stability results. \\
	
	As mentioned earlier to get Fourier estimate of vector potentials over arbitrary large balls we need to use Vessella's analytic cotinuation argument (see \cite{vassella}) to $\eqref{oncone}$.
	To do so, we define \[f_k(t,x)=\widehat{\mc{A}}(kt,kx)\qd\m{ for }k>0\m{ and }(t,x)\in\R^{n+1}.\]   
	Then $f_k$ is an analytic function satisfying the following estimate for multi-index $\gamma$
	\begin{align}      
	&\qd|\partial_{(t,x)}^{\gamma}f_k(t,x)|=\ |\partial_{(t,x)}^{\gamma}\widehat{\mc{A}}(kt,kx)|,\nn\\
	&=\left|\int_{\R^{1+n}}e^{-ik(s,y)\cdot (t,x)}(-i)^{|\gamma|}k^{|\gamma|}(s^2+|y|^2)^{\f{|\gamma|}{2}}\mc{A}(s,y)\ dsdy\right|, \nn \\
	&\le(2T^2)^{\f{|\gamma|}{2}}k^{|\gamma|}\int_{\R^{1+n}}|\mc{A}(s,y)|\ dsdy,\qd (\m{as, diam}(\Omega)< T)\nn \\
	&\le C_*\ (2T^2)^{\f{|\gamma|}{2}}\ k^{|\gamma|}=C_*(\rt{2}T)^{|\gamma|}|\gamma|!\ \f{k^{|\gamma|}}{|\gamma|!}.\qd (\m{using a-priori estimates of}\ \mc{A})\nn
	\end{align} 
	
	\noindent Hence we get, 
	\begin{align}
	&\qd|\partial_{(t,x)}^{\gamma}f_k(t,x)|\le C_*e^{k}\f{|\gamma|!}{(T^{-1})^{|\gamma|}}\ \m{ for }(t,x)\in\R^{n+1}\m{ and multi-index }\gamma. \label{prevessellla1}  
	\end{align}
	Now we appeal to Vessella's conditional stability result  $\cite{vassella}$ to $f_k$ as it satisfies \eqref{prevessellla1}. We have then 
	\begin{align}       
	\|f_k\|_{L^\infty(B(0,1))}\le Ce^{k(1-\theta)}\ \|f_k\|_{L^\infty(\mc{C}_0\cap B(0,1))}^\theta,\qd \m{for some}\ \theta\in(0,1). \label{vassella}
	\end{align}
	Since $\|f_k\|_{L^\infty(B(0,1))}=\|\widehat{\mc{A}}\|_{L^\infty(B(0,k))}$, using Lemma 4 we get from \eqref{vassella}
	\begin{align}  
	\|\widehat{\mc{A}}\|_{L^\infty(B(0,k))}\le Ce^{k(1-\theta)}\left(\frac{1}{\lambda^\delta}+ e^{\beta\lambda}\|\Lambda_1-\Lambda_2\|_{*}\right)^{\theta}.\label{vassella2}
	\end{align}  
	Now using \eqref{vassella2} with the a-priori assumption on the potentials, we express the Sobolev norm of $\mc{A}$ in terms of the input-output operator 
	
	\begin{align}    
	&\|\mc{A}\|_{H^{-1}(\R^{1+n})}^{\f{2}{\theta}}=\left(\int_{\R^{1+n}}(1+s^2+|y|^2)^{-1}|\widehat{\mc{A}}(s,y)|^2dsdy\right)^{\f{1}{\theta}},\nn\\
	&=\left(\int\limits_{B(0,k)}(1+s^2+|y|^2)^{-1}|\widehat{\mc{A}}(s,y)|^2dsdy+\int\limits_{B(0,k)^c}(1+s^2+|y|^2)^{-1}|\widehat{\mc{A}}(s,y)|^2dsdy\right)^{\f{1}{\theta}},\nn\\
	&\le C\left(k^{n+1}\|\widehat{\mc{A}}\|_{L^\infty(B(0,k))}^2+\f{1}{k^2}\right)^{\f{1}{\theta}},\nn\\	
	&\le C\left( k^{\f{n+1}{\theta}}e^{\f{2k(1-\theta)}{\theta}}(\frac{1}{\lambda^{2\delta}}+ e^{\beta\lambda}\|\Lambda_1-\Lambda_2\|_{*}^2)+\f{1}{k^{\f{2}{\theta}}}\right),\qd(\m{using}\ \eqref{vassella2})\nn\\ 
	&\le C\left(\underset{I}{\underbrace{\f{k^{\f{n+1}{\theta}}e^{\f{2k(1-\theta)}{\theta}}}{\ld^{2\delta}}}}
	+\underset{II}{\underbrace{k^{\f{n+1}{\theta}}e^{\beta\ld+\f{2k(1-\theta)}{\theta}}\|\Lambda_1-\Lambda_2\|_{*}^2}}+\underset{III}{\underbrace{\f{1}{k^{\f{2}{\theta}}}}}\right).\label{last1}
	\end{align}
	To make (I) and (III) comparable we need to choose large enough $k$ such that $$\ld=e^{\f{k(1-\theta)}{\theta\delta}}k^{\f{n+3}{2\theta\delta}}.$$ Then (II) becomes \[{k^{\f{n+1}{\theta}}e^{\beta(e^{\f{k(1-\theta)}{\theta\delta}}k^{\f{n+3}{2\theta\delta}})+\f{2k(1-\theta)}{\theta}}\|\Lambda_1-\Lambda_2\|_{*}^2}.\]
	There exists $L_{k_0}>0$ which depends on $\theta>0,\delta>0,\ld_0>0$ only such that 
	\begin{align}k^{\f{n+1}{\theta}}e^{\beta(e^{\f{k(1-\theta)}{\theta\delta}}k^{\f{n+3}{2\theta\delta}})+\f{2k(1-\theta)}{\theta}}\|\Lambda_1-\Lambda_2\|^2\le e^{e^{L_{k_0}k}}\|\Lambda_1-\Lambda_2\|_{*}^2.\label{kchoice}
	\end{align}
	
	\noindent Now we choose $k=\f{1}{L_{k_0}}\log\big|\log\|\Lambda_1-\Lambda_2\|_{*}\big|$. Then from \eqref{last1} and \eqref{kchoice}
	\begin{align}    
	&\|\mc{A}\|_{H^{-1}(\mathbb{R}^{n+1})}^{\frac{2}{\theta}} \le C\left(\|\Lambda_1-\Lambda_2\|_{*}
	+\left(\log|\log\|\Lambda_1-\Lambda_2\|_{*}|\right)^{\f{-2}{\theta}}\right).\label{last0}
	\end{align}    
	Since we need to satisfy $\ld\ge\ld_0$, validity of our above choice of k depends on smallness of the input-output operator, that is when $\|\Lambda_1-\Lambda_2\|_*\le c_*$ (where $c_*$ depends only on $\ld_0$). If it is not the case that is $\|\Lambda_1-\Lambda_2\|_*>c_*$ we can do the following  
	\begin{align}    
	\|\mc{A}\|_{H^{-1}(\mathbb{R}^{n+1})}^{\frac{2}{\theta}} \le C\|\mc{A}\|_{L^\infty(\mathbb{R}^{n+1})}^{\frac{2}{\theta}}\le \f{C}{c_*}{c_*}&\le \f{C}{c_*}\|\Lambda_1-\Lambda_2\|_{*}\ \label{last2}
	\end{align}
	So in both of the cases discussed above we get from \eqref{last0} and \eqref{last2}
	\begin{align}  \|\mc{A}\|_{H^{-1}(\mathbb{R}^{n+1})} \le C\left(\|\Lambda_1-\Lambda_2\|_*^{\f{\theta}{2}}+\big|\log|\log\|\Lambda_1-\Lambda_2\|_{*}|\big|^{-1}\right).\label{last3} 
	\end{align}
	We can translate the above norm estimates for much stronger Sobolev norms using a convexity argument.
	
	\vspace*{-7mm}	 
	\begin{corollary}{3.} 
		For some $ \mu_1,\mu_2,\kappa_1,\kappa_2\in(0,1)$ and $C>0$ we have 
		\[\|\mc{A}\|_{L^\infty{(\mathbb{R}^{n+1})}}\le C\left(\|\Lambda_1-\Lambda_2\|_{*}^{\mu_1}+\big|\log|\log\|\Lambda_1-\Lambda_2\|_{*}|\big|^{-\mu_2}\right).\]
		\[\m{and, }\|\mc{A}\|_{H^{1}{(\mathbb{R}^{n+1})}}\le C\left(\|\Lambda_1-\Lambda_2\|_{*}^{\kappa_1}+\big|\log|\log\|\Lambda_1-\Lambda_2\|_{*}|\big|^{-\kappa_2}\right).\]
	\end{corollary}     
	\begin{proof}  There exists $\eta\in(0,1)$ such that, $\f{n+1}{2}+\f{\al}{2}=\eta(\f{n+1}{2}+\al)+(1-\eta)(-1).$ \\
		Since $\|\mc{A}\|_{H^{\f{n+1}{2}+\al}(\R^{n+1})}\le C$ 
		using Sobolev embedding and logarithmic convexity of Sobolev norms we obtain 
		\begin{align} 
		\|\mc{A}\|_{L^\infty{(\mathbb{R}^{n+1})}}\nn&\le C\|\mc{A}\|_{H^{\f{n+1}{2}+\f{\al}{2}}(\R^{n+1})}\\
		&\le C\|\mc{A}\|_{H^{\f{n+1}{2}+\al}(\R^{n+1})}^{\eta}\|\mc{A}\|_{H^{-1}(\mathbb{R}^{n+1})}^{1-\eta}\nn\\
		&\le C\left(\|\Lambda_1-\Lambda_2\|_{*}^{\f{\theta}{2}}+\big|\log|\log\|\Lambda_1-\Lambda_2\|_{*}|\big|^{-1}\right)^{1-\eta}.\ (\m{using }\eqref{last3})\nn
		\end{align}
		Hence there exists $\mu_1,\mu_2\in(0,1)$ such that 
		\begin{align} 
		\|\mc{A}\|_{L^\infty{(\mathbb{R}^{n+1})}}\le C\left(\|\Lambda_1-\Lambda_2\|_{*}^{\mu_1}+\big|\log|\log\|\Lambda_1-\Lambda_2\|_{*}|\big|^{-\mu_2}\right)\ \mu_1,\mu_2\in(0,1).\label{LAST1}
		\end{align}	
		Similarly one can get	
		\begin{align}
		\|\mc{A}\|_{H^{1}{(\mathbb{R}^{n+1})}}\le C\|\mc{A}\|_{H^{-1}(\mathbb{R}^{n+1})}^{1-\eta}\le C\left(\|\Lambda_1-\Lambda_2\|_{*}^{\kappa_1}+\big|\log|\log\|\Lambda_1-\Lambda_2\|_{*}|\big|^{-\kappa_2}\right).\ \label{LAST3}
		\end{align}
		
	\end{proof}
	
	\subsection{\tb{{Stability estimate for scalar potential}}} 
	Proving stability of scalar potentials will be slightly different from that of vector potentials. In the process of getting integral identity and estimates we divided \eqref{estimate} by large enough $\ld$ which made the scalar potential term disappear. Now we will use explicit uniform norm estimate for vector potentials rather than dividing by $\ld$. So we have to make necessary changes in Lemma 2. Then getting suitable light ray transform of scalar potential or arriving at Fourier estimates of scalar potential  will be same as before. In this section, using uniform norm estimates for the vector potentials and Vessella's analytic continuation argument, Fourier estimate of scalar potential over arbitrary large balls will be shown resulting in log-log-log stability of scalar potential in terms of input-output operator.\\

	\begin{proof}
		We have	$$\mathcal{L}_{\mathcal{A}_1,q_1}u(t,x)=(2\mathcal{A}\cdot(\partial_t,-\nabla_x)u_2+\tilde{q}u_2)(t,x). $$
		Now using H\"older inequality, a-priori bounds of $\mc{A}_i,q_i$ and properties of geometric optics solutions from \eqref{CGO1},\eqref{CGO2},\eqref{derivativeCGO1} and \eqref{CGO4} we have		
		\begin{align}
		\|e^{-(t+x\cdot\og)}\mathcal{L}_{\mathcal{A}_1,q_1}u||_{L^2(Q)}^2\le  C(\ld^2\|\mc{A}\|_{L^\infty(Q)}^2+1)\|\phi\|_{H^3(\R^n)}^2\label{new1}.
		\end{align}
		Hence we obtain from \eqref{finaltime}, \eqref{boundarygreen2} and \eqref{new1}
		\begin{align}\mc{K}\le C(\lambda^2\|\mc{A}\|_{L^\infty(Q)}^2+1+ e^{\beta\lambda}\|\Lambda_1-\Lambda_2\|_{*}^2)\|\phi\|_{H^3(\mathbb{R}^n)}^2. 
		\end{align}
		
		\noindent Now we proceed as before to have estimates over light ray transform of scalar potential and in this situation it will include uniform norm of vector potentials on $Q$. We see 
		\begin{align} 
		&\left|\int\limits_{Q}(2\mathcal{A}\cdot(\partial_t,-\nabla_x)u_2+\tilde{q}u_2)\overline{v}\ dxdt\right|\le C\left(\sqrt{\frac{\mc{K}}{\lambda}}\|\phi\|_{H^3(\mathbb{R}^n)}+ e^{\beta\lambda}\|\Lambda_1-\Lambda_2\|_{*}\|\phi\|_{H^3(\mathbb{R}^n)}^2\right),\nn\\
		&\m{or, }\left|\int_{Q}\tilde{q}(t,x)u_2(t,x)\overline{v(t,x)}\ dxdt\right|\le C(\ld\|\mc{A}\|_{L^\infty(Q)}+\f{1}{\rt{\ld}}+ e^{\beta\lambda}\|\Lambda_1-\Lambda_2\|_{*})\|\phi\|_{H^3(\mathbb{R}^n)}^2,\nn\\
		&\m{or, }\left|\int_{Q}\tilde{q}(t,x)|\phi(x+t\og)|^2  e^{-\int_{0}^{t}(1,-\og)\cdot\mc{A}(s,x+(t-s)\og)ds}dxdt\right|\nn\\
		&\qd\qd\qd\qd\ \ \ \ \qd\qd\qd\ \ \ \qd\le C\left(\ld\|\mc{A}\|_{L^\infty(Q)}+\f{1}{\rt{\ld}}+ e^{\beta\lambda}\|\Lambda_1-\Lambda_2\|_{*}\right)\|\phi\|_{H^3(\mathbb{R}^n)}^2. \label{R1}
		\end{align}
		Now using mean value theorem and a-priori bounds of vector potentials we get
		\begin{align}
		\big|e^{-\int_{0}^{t}(1,-\og)\cdot\mc{A}(s,x+(t-s)\og)ds}-1\big|\le C\|\mc{A}\|_{L^{\infty}(Q)}\qd\m{ for all }\ t\in[0,T].\label{M.V.T}
		\end{align}
		Thus we use \eqref{R1} and \eqref{M.V.T} to get
		\begin{align}  
		\left|\int_{Q}\tilde{q}(t,x)|\phi(x+t\og)|^2\ dxdt\right|&\le \left|\int_{Q}\tilde{q}(t,x)|\phi(x+t\og)|^2\left(e^{-\int_{0}^{t}(1,-\og)\cdot\mc{A}(s,x+(t-s)\og)ds}-1\right)dxdt\right|\nn\\
		& \qd +\left|\int_{Q}\tilde{q}(t,x)|\phi(x+t\og)|^2e^{-\int_{0}^{t}(1,-\og)\cdot\mc{A}(s,x+(t-s)\og)ds}dxdt\right|,\nn\\
		&\le C\left(\ld\|\mc{A}\|_{L^\infty(Q)}+\f{1}{\rt{\ld}}+ e^{\beta\lambda}\|\Lambda_1-\Lambda_2\|_{*}\right)\|\phi\|_{H^3(\mathbb{R}^n)}^2. \label{R2}
		\end{align}
		Following similar steps as in Lemma 3 and Corollary 2 we have $\gamma,\delta>0$ such that
		\begin{align}
		\left|\int_{0}^{T}\tilde{q}(s,x-ws)ds\right|\ \m{and,}\ \|\hat{\tilde{q}}\|_{L^\infty(\mc{C})}\le C\left(\ld^{\gamma}\|\mc{A}\|_{L^\infty(Q)}+\f{1}{\ld^{\delta}}+ e^{\beta\lambda}\|\Lambda_1-\Lambda_2\|_{*}\right).\label{R} 
		\end{align}
		
		\noindent
		Using $\eqref{R}$ and Vessella's analytic continuation result as done in $\eqref{vassella2}$ we get for $k$ large 
		\begin{align}
		&\|\widehat{\tilde{q}}\|_{L^\infty(B(0,k))}\ \le Ce^{k(1-\theta)}(\ld^{\gamma}\|\mc{A}\|_{L^\infty(Q)}+\f{1}{\ld^{\delta}}+ e^{\beta\lambda}\|\Lambda_1-\Lambda_2\|)^{\theta}. \label{overball}\\
		\m{Then, }\qd&\|\tilde{q}\|_{H^{-1}{(\R^{n+1})}}^{\frac{2}{\theta}}\le C\left(k^{\f{n+1}{\theta}}e^{\f{2k(1-\theta)}{\theta}}\left(\ld^{2\gamma}\|\mc{A}\|_{L^\infty(Q)}^2+\f{1}{\ld^{2\delta}}+e^{\beta\ld}\|\Lambda_1-\Lambda_2\|^2\right)+k^{-\frac{2}{\theta}}\right).\label{sc1}
		\end{align}
		Now as in $\eqref{last1}$ choose $\ld=e^{\f{k(1-\theta)}{\theta\delta}}k^{\f{n+3}{2\theta\delta}}$. For some $L_{k_0}>0$, using $\eqref{LAST1}$ R.H.S of 
		$\eqref{sc1}$ can be dominated by the following term 
		\begin{align}
		e^{L_{k_0}k}\|\Lambda_1-\Lambda_2\|_{*}^{2\mu_1}+e^{L_{k_0}k}\big|\log|\log\|\Lambda_1-\Lambda_2\|_{*}|\big|^{-2\mu_2}+e^{e^{L_{k_0}k}}\|\Lambda_1-\Lambda_2\|_{*}^2+k^{-\frac{2}{\theta}}.\label{sc2}
		\end{align}
		Choose $e^{L_{k_0}k}=\log|\log\|\Lambda_1-\Lambda_2\|_{*}|^{\mu_2}$. That is, $k=\f{\mu_2}{L_{k_0}}\log\log\big|\log\|\Lambda_1-\Lambda_2\|_{*}\big|$.\\
		Using the choice of $k$ above and $\mu_2<1$, R.H.S of $\eqref{sc2}$ can have following upper bound    
		\begin{align} {{\|\Lambda_1-\Lambda_2\|_{*}^{2\mu_1}\big|\log|\log\|\Lambda_1-\Lambda_2\|_{*}|\big|^{\mu_2}}}+{\big|\log|\log\|\Lambda_1-\Lambda_2\|_{*}|\big|^{-\mu_2}}\nn\\
		+{{\|\Lambda_1-\Lambda_2\|_{*}^{2}\left|\log\|\Lambda_1-\Lambda_2\|_{*}\right|}}+\big|\log|\log\|\Lambda_1-\Lambda_2\|_{*}|\big|^{-\frac{2}{\theta}}. \label{sc3}
		\end{align}
		Assume $\|\Lambda_1-\Lambda_2\|_{*}$ small enough say $\|\Lambda_1-\Lambda_2\|_{*}<c_*$ such that 
		\begin{align}
		&a)\ \m{choice of k above is valid}.\label{condition1}\\
		&b)\ \|\Lambda_1-\Lambda_2\|_{*}^{2\mu_1}\le C(\log|\log\|\Lambda_1-\Lambda_2\|_{*}|)^{-2\mu_2}\ (\m{for }\al>0,\ \lim\limits_{x\to 0+}x^{\al}\log|\log x|=0).\label{condition2}\\
		&c)\ \left(\log|\log\|\Lambda_1-\Lambda_2\|_{*}|\right)^{-\mu_2}\le (\log\log\left|\log\|\Lambda_1-\Lambda_2\|_{*}\right|)^{-\mu_2}\ (\m{as, }\log x \le x).\label{condition3}\\
		&d)\ \|\Lambda_1-\Lambda_2\|_{*}\left|\log\|\Lambda_1-\Lambda_2\|_{*}\right|\le C\ \ (\m{for }\al>0,\ \lim\limits_{x\to 0+}x^{\al}\log x=0).\label{condition4}
		\end{align} 
		Combining \eqref{condition1},\eqref{condition2},\eqref{condition3} and \eqref{condition4} we dominate \eqref{sc2} by the following term \begin{align}\|\Lambda_1-\Lambda_2\|_{*}^{\nu_1}+\left(\log\log\left|\log\|\Lambda_1-\Lambda_2\|_{*}\right|\right)^{-\nu_0}\m{ for some }\nu_1,\nu_0>0.\label{scalar1}
		\end{align} 
		When $\|\Lambda_1-\Lambda_2\|_{*}\ge c_*$ we may proceed as $\eqref{last2}$ to get similar estimates as \eqref{scalar1}. Hence by Sobolev embedding and logarithmic convexity of Sobolev norms we get 
		\begin{align}
		\|\tilde{q}\|_{L^\infty{(\R^{n+1})}}\le C\left(\|\Lambda_1-\Lambda_2\|_{*}^{\nu_1}+\Big|\log\big|\log|\log\|\Lambda_1-\Lambda_2\|_{*}|\big|\Big|^{-\nu_2}\right)\qd \ \m{for some } \nu_1,\nu_2 >0.\label{sc4}
		\end{align}
		Our goal was to establish norm estimate for $q$. We observe   $$q(t,x)=\tilde{q}(t,x)-\left(\partial_tA_0-\sum_{k=1}^{n}\partial_{x_k} A_k\right)+\left(|A_{1,0}|^2-|A_{2,0}|^2\right)(t,x)+\sum\limits_{k=1}^{n}{\left(|A_{2,k}|^2-|A_{1,k}|^2\right)(t,x)}.$$ 
		So we can write 
		\begin{center}
			$\|q\|_{L^\infty{(\R^{n+1})}}\le C\left(\|\tilde{q}\|_{L^\infty{(\R^{n+1})}}+\|\mc{A}\|_{H^{1}{(\R^{n+1})}}+\|\mc{A}\|_{L^\infty{(\R^{n+1})}}\right). \nn\label{sc5}$
		\end{center}
		We combine \eqref{LAST3},\eqref{LAST1} and \eqref{sc4} to obtain
		\[\|q\|_{L^\infty{(\R^{n+1})}}\le C\left(\|\Lambda_1-\Lambda_2\|_{*}^{\alpha_1}+\Big|\log\big|\log|\log\|\Lambda_1-\Lambda_2\|_{*}|\big|\Big|^{-\alpha_2}\right)\m{ for some }\alpha_1,\alpha_2 >0. \nn\] 
		This shows stability of scalar potentials from input-output operator.
	\end{proof}               
	
	\noindent \tb{Acknowledgements:} The author expresses his gratitude to Venkateswaran P. Krishnan for suggesting the problem and many fruitful discussions and comments on this project.

\end{document}